\documentclass[12pt,regno]{amsart}

\usepackage{overpic}
\usepackage{euscript,graphicx,color}
\usepackage{amsmath}
\usepackage{amssymb}
\usepackage{amsfonts}
\usepackage{mathrsfs}

\newtheorem{theo}{Theorem}[section]
\newtheorem{theo-app}{Theorem}[section]

\newtheorem{theorem}[theo]{Theorem} %dopisalem, numeruje twierdzenia wg rozdzialow.
\newtheorem{corollary}[theo]{Corollary}
\newtheorem{proposition}[theo]{Proposition}

\theoremstyle{definition}
\newtheorem{example}[theo]{Example}
\newtheorem{remark}[theo]{Remark}
\newtheorem*{remark*}{Remark}
\newtheorem{definition}[theo]{Definition}

\newtheorem{notation}[theo]{Notation}

\numberwithin{equation}{section}
\renewcommand{\hat}{\widehat}

\newcommand{\sR}{\mathscr{R}}

\DeclareMathOperator{\cl}{cl}

\newcommand{\R}{\mathbb{R}}
\newcommand{\N}{\mathbb{N}}
\newcommand{\C}{\mathbb{C}}
\newcommand{\Z}{\mathbb{Z}}

\DeclareMathOperator{\var}{var}

\DeclareMathOperator{\Int}{Int}

\DeclareMathOperator{\diam}{diam}

\DeclareMathOperator{\dist}{Dist}

\DeclareMathOperator{\Crit}{{\rm Crit}}

\DeclareMathOperator{\Comp}{{\rm Comp}}

\DeclareMathOperator{\sep}{{\rm sep}}

\DeclareMathOperator{\tree}{{\rm tree}}
\DeclareMathOperator{\Const}{{\rm Const}}

\DeclareMathOperator{\HD}{{\rm HD}}
\DeclareMathOperator{\hyp}{{\rm hyp}}

\def\cH{\mathscr{H}}

\def\cP{\mathscr{P}}

\def\e{\varepsilon}

\def\ov{\overline}
 \def\dist{{\rm {dist}}}

\DeclareMathSymbol{\varnothing}{\mathord}{AMSb}{"3F}
\renewcommand{\emptyset}{\varnothing}

\author[F. Przytycki]{Feliks Przytycki$^\dag$} \address{Institute of Mathematics, Polish Academy of Sciences, ul. \'{S}niadeckich 8, 00-656 Warszawa, Poland}
\email{feliksp@impan.pl}

\thanks{Supported partially by National Science Centre, Poland, Grant OPUS21 "Holomorphic dynamics, fractals, thermodynamic formalism",
2021/41/B/ST1/00461.}

\begin{document}

\date{\today}

\title[McMullen's and geometric pressures]{McMullen's and geometric pressures and approximating Hausdorff dimension of Julia sets from below}



\maketitle

\begin{abstract} We introduce new variants of the notion of geometric pressure for rational functions on the Riemann sphere, including non-hyperbolic functions, in the hope some of them occur useful to achieve a fast approximation from below of the hyperbolic Hausdorff dimension of Julia set.
\end{abstract}

\tableofcontents

\section{Introduction}\label{Introduction}

Iterating the action of a mapping $f$, (roughly) increasing distances, for example a rational function on a neighbourhood of its "chaotic" invariant Julia set, leads its small subsets to large ones, roughly preserving shapes (dynamical "escalator"). So a long time behaviour under the action of $f$, provides an insight into the local structure of $J(f)$, here, in complex dimension 1, into its Hausdorff dimension characteristic.

%To understand the dynamic phase space one distributes masses on it according to potentials, on small sets according to their sums over times leading them to large scale.

A tool is a \emph{geometric pressure} with respect to the potential
$\phi=\phi_t = -t\log |f'|$ on Julia set, here for $t >0$.
The pressure (free energy) can be defined for any $\phi$  in a variational way
\begin{equation}\label{varpres}
P_{\var}(f,\phi)=\sup_\mu (h_\mu(f)+\int \phi\, d\mu),
\end{equation}
supremum over all $f$-invariant probability measures $\mu$ on $J(f)$. $h_\mu(f)$ means the measure-theoretical
(Kolmogorov's) entropy and $\mu$ for which the supremum is attained (if it exists) an equilibrium state.
Below we provide different definitions.
There is an analogy with equilibria in statistical physics, e.g. Ising model of ferromagnetic, where equilibria
are distributed on the space of all configurations of + or - over $\Z^2$ with a potential depending on a hamiltonian function  expressed in terms of interactions  between elements of the configuration.

 The founders of applications in dynamics are in particular Y. Sinai, D. Ruelle and R. Bowen (SRB measures). Here we consider forward trajectories so the "configurations" are over $\N$.  In particular a geometric application with the use of $\phi_t$, hence  $\exp S_n(\phi_t)=|(f^n)'|^{-t}$,
is to relate (roughly) an equilibrium measure (a mass) of a disc of diameter $|(f^n)'|^{-1}$ with this diameter, for each $t$ up to a normalizing coefficient
$\exp nP(f,\phi_t)$, equal to 1 if $t=t_0$ a zero of the pressure $P(f,\phi_t)$, denoted also $P(f,t)$.
This $t_0$ is  called hyperbolic Hausdorff dimension of $J(f)$, $\HD_{\hyp}(f,t)$, defined as the supremum of Hausdorff dimensions
of invariant hyperbolic subsets of $J(f)$.

An introduction to this theory is provided in
\cite{PUbook}. Closer to the content of the paper are  \cite{P-TAMS2} and \cite{PRS}, where geometric  pressure
 for general rational functions was first defined and studied. See also \cite{P-ICM}.

The aim of this note is to introduce more variants of this notion, in particular some, close to McMullen's
pressure defined for hyperbolic rational functions in \cite{McM-HD3}, useful to numerically calculate Hausdorff dimension of the underlying Julia sets, estimate it from below.

The pressure function $t\mapsto P(f,t)$ will occur to be a limit from below of a sequence of functions specific to each notion of the pressure we introduce. Therefore their first zeros will converge to $\HD_{\hyp}(f,t)$ from below. There are two concepts in these notions of pressure. 1. To replace a potential along each trajectory by a
"fuzzy" one, mainly by replacing the value at a point by the infimum of the value in its small neighbourhood, or
the smaller one at one of two sampling points close to it. 2. To restrict the pressure to trajectories not passing too close to the set of critical points.

The key issue we do not address here is how efficient is the calculation of these functions and how fast are the convergences. This will be dealt with in \cite{DGT}.

\

\emph{Acknowledgements.}
A direct stimulus for me to write this note has been listening to Artem Dudko's talks on estimating Hausdorff dimension of Feigenbaum Julia set from below, at IM PAN dynamical systems seminar, \cite{DGT}. I am indebted to him for fruitful discussions and suggestions.

\

\section{Definitions, the statement of a main theorem}\label{definitions}

\

\subsection{Topological pressures}\label{topological pressures}
Start recalling the basic definition of the topological pressure for a continuous transformation of a metric compact space
and real continuous potential, see e.g. \cite{Walters} or \cite{PUbook}.

\begin{definition}[topological pressure via separated sets]\label{topological pressure}
Let $f:X\to X$ be a continuous map and  $\phi:X\to\R$ a continuous real-valued function (potential) on $X$.  Consider for $S_n$ defined below, in \eqref{fuzzy},
\begin{equation}\label{Psep}
P_{\rm {sep}}(f,\phi):=\lim_{\e\to 0}\limsup_{n\to\infty}\frac1n \log \big(\sup_Y \sum_{y\in Y}\exp S_n\phi(y)\big),
\end{equation}
where the supremum is taken over all $(n,\e)$-separated sets $Y\subset X$, that is such $Y$ that for every $y_1,y_2\in Y$ with $y_1\not=y_2$, $\rho_n(y_1,y_2)\ge \e$, where $\rho_n$ is the metric defined by $\rho_n(x,y)=\max\{\rho(f^j(x),f^j(y)): j=0,\dots,n\}$.
\end{definition}

We can slightly modify this definition, defining a \emph{fuzzy pressure} or \emph{inf-pressure} $P^0_{\sep}(f,\phi)$ by
first replacing  in
\begin{equation}\label{fuzzy}
S_n\phi(y):=\sum_{j=0}^{n-1} \phi(f^j(y))
\end{equation}
in \eqref{Psep}, $\phi(x)$ for $x=f^j(y)$ for each $j$, by
$\inf\{\phi(z): z\in B(x,\delta)\}$
thus defining
$$
S^\delta_n\phi(y):=\sum_{j=0}^{n-1}   \inf\{\phi(z): z\in B(f^j(y),\delta)\}
$$
writing $P_{\sep}^\delta(f,\phi)$, and at the end defining
$$
P^0_{\sep}(f,\phi):=\lim_{\delta\to 0}P_{\sep}^\delta(f,\phi).
$$
By the uniform continuity of $\phi$, an easy calculation gives $P_{\sep}(f,\phi)= P^0_{\sep}(f,\phi)$.

A related notion is \emph{tree pressure} (or \emph{Gurevitch pressure}), interesting for a non-invertible $f$,  defined for an arbitrary $z\in X$ by

\begin{equation}\label{Ptree}
P_{\tree}(f,\phi,z)=\limsup_{n\to\infty}\frac{1}{n}\log\sum_{y\in f^{-n}(z)} \exp S_n\phi (y).
\end{equation}
which also can be defined in a "fuzzy" way as a \emph{fuzzy tree pressure},
another name \emph{infimum tree pressure},
\begin{equation}\label{P0}
P^0_{\tree}(f,\phi,z):=\lim_{\delta\to 0} P^\delta_{\tree}(f,\phi,z)
\end{equation}
where
\begin{equation}\label{Pdelta}
P^\delta_{\tree}(f,\phi,z):=\limsup_{n\to\infty}\frac{1}{n}\log\sum_{y\in f^{-n}(z)} \exp S^\delta_n\phi (y).
\end{equation}
Later on we shall use the following easy, \cite[Chapter 4]{PUbook},

\smallskip

\begin{remark}\label{easy}
Suppose $f:X\to X$ is open distance expanding topologically exact, namely foe every open $U\subset X$ there exists $n\in\N$ such that $f^n(U)=X$, and $\phi:X\to\R$ be continuous.
Then all the above pressures coincide and are independent of $z$. So we can denote them just by $P(f,\phi)$.
\end{remark}

\medskip

\subsection{Geometric pressures}\label{geometric pressures}
From now on we shall consider a rational transformation of the Riemann sphere  $f:\ov\C \to \ov\C$ and its restriction to  the Julia set $J(f)$.
We shall consider geometric potentials $\phi=\phi_t= -t\log|f'|$ for $t>0$.
Note that in this case we can write $\exp S_n\phi_t$ in Definition \ref{topological pressure}, in the form
$|(f^n)'|^{-t}$. The derivative $f'$ will be considered only with respect to  the spherical Riemann metric. We shall use only its absolute value $|f'|$ so there will be no ambiguity caused by its argument. The points $x\in\ov\C$ where $f'(x)=0$ are called \emph{critical} and their set denoted by  $\Crit(f)$. For $c\in \Crit(f)$ where in the complex plane coordinates $f(z)= a(z-c)^\nu +  b(z-c)^{\nu+1}+...$, with $a\not=0$, $\nu=\nu(c)$ is called the multiplicity of $f$ at $c$. We shall consider also the \emph{post-critical set}, \;
${\rm PC}(f):=\bigcup_{n=1}^\infty f^n(\Crit(f))$.

\smallskip

If the forward trajectory of no critical point accumulates at $J(f)$, that is there are no $f$-critical points in $J(f)$, nor attracted to parabolic periodic orbits, then $f|_{J(f)}$ is open expanding. Another term for this is
\emph{hyperbolic}).  It means there exist $C>0,\lambda>1$ such that $|(f^n)'(z)|>C\lambda^n$ for all $z\in J(f)$ and $n\in\N$. This  is an easy case, covered by Remark \ref{easy}.

\bigskip

We shall consider from now on a general case with critical points whose forward trajectories can accumulate at $J(f)$.
The above definitions of $P_{\tree}(f,\phi,z)$ and $P^0_{\tree}(f,\phi,z)$ above make sense
%for $f$ restricted to the Julia set $J(f)$ and $\phi_t:J(f)\to\R$,
even though $\phi_t$ is infinite at critical points (yielding $\infty$) and for $z$ outside $J(f)$. The pressure $P_{\tree}(f,\phi_t,z)$ does not depend on $z\in\ov\C$ for non-exceptional $z$ (not in or fast accumulated by ${\rm PC}(f)$, see Definition \ref{non-exceptional}). See \cite{P-TAMS2} and \cite{PRS}.
So we can denote it just by $P_{\tree}(f,t)$. The independence of $P^0_{\tree}(f,\phi,z)$ of non-exceptional $z$ also holds, via $P_{\hyp}(f,t)$ and $P_{\tree}(f,t)$, see Main Theorem \ref{ss:MT}, item 1.

\smallskip

To fix notation for $\phi=\phi_t=-t\log |f'|^{-t}$ let us
rewrite:

\begin{notation}[geometric tree pressure]\label{PtreeW}
%\ inf-pressure is defined by

\begin{equation}\label{gtp}
P_{\tree}(f,t,z)=P_{\tree}(f,\phi_t,z):=\limsup_{n\to\infty}\frac{1}{n}\log \sum_{v\in f^{-n}(z)}\Pi_n(t,v)
\end{equation}
where
\begin{equation}
\Pi_n(t,v):=\prod_{k=1}^{n}  |f'(f^{n-k}(v))|^{-t}\}\bigr)=|(f^n)'(v)|^{-t}.
\end{equation}
\end{notation}
and
\begin{notation}[geometric fuzzy tree pressure]\label{gftp}
\begin{equation}\label{eq:gftp}
P_{\tree}^0(f,t,z):=\lim_{\delta\to 0}
\limsup_{n\to\infty}\frac{1}{n}\log\sum_{v\in f^{-n}(z)} \prod_{k=0}^n \inf\{|f'(y)|^{-1}: y\in B(f^{n-k}(v),\delta)\}.
\end{equation}
\end{notation}

We shall introduce also a new notion
\begin{definition}[pullback infimum tree pressure]\label{def:pullinf}
\begin{equation}\label{Pi-inf}
P_{\tree}^{\rm pullinf}(f,t,z):=\lim_{r\to 0}\limsup_{n\to\infty}\frac{1}{n}\log\sum_{v\in f^{-n}(z)} \Pi^{\rm pullinf}_n(t,v)
\end{equation}
where
\begin{equation}\label{Pi-inf2}
\Pi^{\rm pullinf}_n(t,v):=\prod_{k=1}^{n} \inf\{|f'(y)|^{-t} :
  y\in \Comp_{f^{n-k}(v)} f^{-k}(B(f^n(v),r))\}
\end{equation}
where $\Comp_x$   means the component of a set containing $x$.

The limit for $r\to 0$ exists since the family under it is monotone increasing as $r\to 0$ because infima in
\eqref{Pi-inf2} are taken on shrinking sets.
\end{definition}

%and  $W_{k,f^{n-k}(v),r}:=\Comp_{f^{n-k}(v)} f^{-k}(B(z,r))$.

%We will show for every $z\in J(f)$ that $P_{\tree}^{\inf}(t,f,z)$ does not depend on $r$ for $r\le r_0$ small enough depending only on $f$.

\bigskip

Let the following definition of $P(f, -t\log |f'|)$, called \emph{hyperbolic pressure}  be considered as a default one, to be denoted $P(f,t)$, see e.g. \cite{P-ICM}:

\begin{definition}[Hyperbolic pressure]\label{hyperbolic pressure}
$$
P(f,t)=P_{\hyp}(f,t):= \sup_{X\in\cH(f,J(f))} P(f|_{X},-t\log|f'|) ,
$$
where $\cH(f,J(f))$ is defined as the space of all compact forward $f$-invariant (that is $f(X)\subset X$)
hyperbolic topologically exact subsets of $J(f)$ and repelling for $f$, that is if a forward trajectory of a point is in a sufficiently small neighbourhood of $X$ then it is entirely in $X$. The repelling assumption can be omitted without influencing the resulting $P_{\hyp}(f,t)$.
\end{definition}

\subsection{Hausdorff dimension}\label{ss:HD}
From this definition and Bowen's formula for every hyperbolic set $X$, saying that
Hausdorff dimension $\HD(X)$ is the only zero of the function $t\mapsto P(f|_X,t)$, see \cite{Bowen},
it immediately follows that:

%\noindent \cite[Corollary 12.5.12]{PU} in the complex case) the following

\begin{proposition}\label{hypdim}{\rm (Generalized Bowen's formula)}%, compare \eqref{HDformula})}
The first zero $t_0$ of $t\mapsto P_{\hyp}(t)$ is equal
to the hyperbolic dimension $\HD_{\hyp} (J(f))$, defined by
$$
\HD_{\hyp} (J(f)):=\sup_{X\in\cH(f,K)} \HD(X).
$$
%where the supremum is taken over all compact forward $f$-invariant
%hyperbolic subsets of $K$, repellers in $\R$.
\end{proposition}

\smallskip

Note that this zero exists, since all $P(f|_{X},-t\log|f'|)$ are decreasing, hence their limit
$P_{\hyp}(f,t)$ is decreasing and  $\HD(X)\le 2$ since $X\subset \ov\C$ of dimension 2.

\

\subsection{Main Theorem}\label{ss:MT}

We shall prove in this paper the following (some definitions to be provided later on)

\begin{theorem}[Main Theorem]\label{equality-sup-pressure}
 \  Let $f:\ov\C\to\ov\C$ be a rational function \emph{backward uniformly asymptotically stable}.
  Then

  1.  $P(f,t)=P_{\tree}(f,t,z)=P_{\tree}^0(f,t,z)$ for all non-exceptional $z\in\ov\C$

  2. $P(f,t)=\widehat{P_{\rm McM}}(f,t)=P^0_{\rm McM}(f,t)\le P_{\rm McM}(f,t)  $ (restricted and fuzzy  McMullen's pressures) provided
  for  the puzzle structure $\cP$ in the definition of McMullen's pressures the diameters of the puzzle pieces  of the renormalizations ${\sR}^N(\cP)$ tend uniformly to 0 as $N\to\infty$.

%  $f$ is \emph{backward uniformly asymptotically stable} with a constant $r_0$ and
%for $f$ for which a puzzle structure, of all pieces with sufficiently small diameters, exists. % and appropriate sample points in the puzzle- pieces of all generations are chosen.

3. $P(f,t)=P_{\tree}^{\rm pullinf}(f,t)$ for all $r$ small enough, provided $f$ is \emph{backward uniformly asymptotically stable}.
Moreover for all $r$ small enough, all non-exceptional $z$ and every backward trajectory $(z_n)$ of $z$ (that is $f(z_n)=z_{n-1}$ for all $n\in\N$, and $z_0=z$)
\begin{equation}\label{tele} \frac1n \log
\Bigl(\frac     {\Pi^{\rm pullinf}_n (t,z_n)}    { |(f^n)'(z_n)|^{-t}} \Bigr)
\to 0
\end{equation}
as $n\to\infty$, uniformly with respect to the backward trajectory $(z_n)$, where $z_0=z$.
\end{theorem}

%We used here (and already before) and will use in the sequel the notation
%Recall that $\Comp_x f^{-k}(W)$ denotes the component of
%the set $f^{-k}(W)$ for any connected set $W\subset \ov\C$ in particular
%for $W=B(f^k(x),\tau)$, containing $x$, in other words for the \emph{ pullback} for $f^k$ containing   $x\in \ov\C$, of the disc $W=B(f^k(x),\tau)$ with the origin at $f^k(x)$ and an arbitrary radius $\tau$ in the spherical metric.

%We shall use also for $\Comp_x f^{-k}(B(f^k(x),\tau))$ the notation$W_{k,x,\tau}$.

%Note that $|(f^n)'|^{-t}(v)=\exp S_n (-t\log |f'|(v)$ making it compatible with notation in
%Definition \ref{topological pressure}.
\smallskip

The equalities in the item 1. say in particular that the pressures there do not depend on $z$ for all non-exceptional $z$
%see the definitions below.

%Moreover

%\begin{theorem}\label{equality-McMullen}
%Let $f:\ov\C\to\ov\C$ be a rational function \emph{backward uniformly asymptotically stable},
%for which a puzzle structure exists.
%Then for every non-exceptional $z\in J(f)$, McMullen's pressure, chosen appropriate sample points in the puzzle- pieces of all generations, satisfies
%$$
%P_{\rm McM }(f,t,z)=P_{\rm sup }(f,t,z)=P_{\tree}(f,t,z)=P(f,t),
%$$
%\end{theorem}

\begin{definition}\label{back-stab}
 $f$ is said to be \emph{backward uniformly asymptotically stable}, abbreviated: (buas), if
there exists $r_0>0$ such that for each  $z\in J(f)$, Julia set,  diameters
of all \emph{pullbacks} of $B(z,r_0)$ (namely: components of preimages) for $f^n$ tend to 0 uniformly fast, with respect to $z$
and to the pullback, as $n\to\infty$.
\end{definition}

Notice that this property is hereditary, that is if it holds for $r$, then it holds for every $\tilde r \le r$.

\begin{definition}\label{non-exceptional}
We call a point $z\in \ov\C$ \emph{non-exceptional} if
for each $\epsilon>0$ and $n$ large enough
$B(z, \exp -n\epsilon)$ is disjoint from $f^j(\Crit(f))$ for all $j=1,...,n$.  In particular
$z$ is not \emph{post-critical} that is $z\notin {\rm PC}(f)=\bigcup_{n=1}^\infty f^n(\Crit(f))$.
The other points are called \emph{exceptional} and the set of all exceptional points is denoted by $E$.

\end{definition}

% denotes the component of the set

%\noindent $f^{-k}(B(z,r))$ containing $f^{n-k}(v)$, in other words \emph{the pullback} of the disc $B(z,r)$, for $f^k$,  containing the point  $f^{n-k}(v)\in f^{-k}(z)$.

%\end{definition}

%\begin{notation}\label{W}
%In the sequel we shall sometimes write $W_{i,v,r}:=\Comp_v f^{-i}(B(f^i(v),r))$ for any $v\in J(f)$ and $i\in \N$ where
%$f^i(v)$ need not be any distinguished point $z$.
%\end{notation}

%\begin{definition}\ McMullen's pressure is defined by

%\end{definition}

%------------------------------ PRZENIESC?

%\begin{definition}
%Geometric pressure can be defined as
%so-called tree pressure
%\begin{equation}\label{Ptree}
%P_{\tree}(t,f,z)=\lim_{n\to\infty}\frac{1}{n}\log\sum_{v\in f^{-n}(z)}  \frac{1}{|(f^n)'(v)|^t}
%\end{equation}
%existing and independent of $z\in\widehat{\mathbb C}\setminus E$, for an exceptional set
It is clear from the definition that the exceptional set $E\subset \widehat{\C}$ has  Hausdorff dimension $0$.
%This is a "dynamical tree" variant of the geometric pressure, first defined in  \cite{P-TAMS2}.
 Sometimes it is comfortable to consider $z\in \ov\C\setminus J(f)$, so close to $J(f)$ that it cannot be postcritical itself neither be accumulated by the postcritical set. Then of course it must be non-exceptional.

\medskip

\begin{remark}
$\limsup_{n\to\infty}$ in the definition \eqref{Ptree}  can be replaced by lim, which always exists, see \cite[Remark 12.5.18]{PUbook}.
 So the limits exist also in the definition of $P^{\rm pullinf}_{\tree}(f,t,z)$ for $z\in J(f)$, provided the property  (buas).
  It will follow easily from the proof of Theorem \ref{equality-sup-pressure}, see \eqref{tele} and \eqref{individual}.
  The same holds for $P^{\rm infW}_{\tree}(f,t,z)$, to be defined in \eqref{PsupW}.
\end{remark}

%---------------------------------------

\

\section{Fuzzy (infimum) and pullback infimum tree pressures}\label{sec:fuzzy-inf}

\begin{proof}[Proof of Theorem \ref{equality-sup-pressure}, the item 1. ]

It follows immediately from known theory and easy observations.

%The inequality $P_{\tree}^{\rm inf}(t,f,z)\le P_{\tree}(t,f, z)$ holds trivially  for every  $z\in J(f)$
%since

 Indeed, the inequality $P_{\hyp}(f,t)\le P_{\tree}(f,t,z)$ follows from the trivial observation
 that for every  $X\in \cH(f,J(f))$ in Definition \ref{hyperbolic pressure} and $z\in X$,
 $$
 P_{\tree}(f|_X,-t\log|f'| \le P_{\tree}(f, -t\log|f'|, z)
 $$ since in the former pressure we count only backward
 trajectories in $X$ of a non-exceptional $z\in X$, whereas in the latter one all in $J(f)$.
 A non-exceptional $z\in X$ exists since $\HD_{\hyp}(J(f))>0$ hence in its definition it is sufficient to consider only $X$ satisfying $\HD(X)$, and $\HD(E)=0$.

 The same reasoning holds for
 $P_{\hyp}(f,t)\le P_{\tree}^0(f,t)$, defined in \eqref{eq:gftp}. Take in account that for every backward trajectory $(z_n)$ of $z$ in $X$ we have
 $\log|f'(z_k)\le \log|f'(v)|+\epsilon$ for every $v\in\ov\C$ with $\rho(x_k,v)\le \delta$,
 for every $\epsilon$ and $\delta$ small enough. This is so because $X$ is disjoint hence, as being compact, bounded away from $\Crit f$.

 The inequality
 $P_{\tree}^0(f,t,z)\le P_{\tree}(f,t, z)$ holds trivially  for every  $z$
since $\inf\{v\in B(z_k,\delta): |f'(y)|^{-1}\}
 \le |f'(z_k)|^{-1}$ for all $k$. The latter yields also the monotone increasing of
 $P_{\tree}^\delta (f,t,z)$ as $\delta\to 0$.

The proof that $P_{\tree}(f,t, z) \le P_{\hyp}(f,t)$ for non-exceptional $z\in J(f)$ is harder, fortunately it is known.
First one assumes that $z$ is non-exceptional and additionally \emph{hyperbolic}, which means by definition that there exists $r>0$ and $\lambda>1$ such that for every disc $B(z,\tau)$ there exists $n\in N$ such that $f^n|_{B(z,\tau)}$ is univalent and $f^n(B(z,\tau))\supset B(f^n(z), r)$ \cite[Proposition 2.1]{PRS}.
An idea of the proof is to capture a large hyperbolic set using a "shadowing".

Next one proves that for two non-exceptional $z^1$ and $z^2$ the equality
$P_{\tree}(f,t, z^1) = P_{\tree}(f,t, z^2)$ holds. See \cite[Theorem 3.3]{P-TAMS2} and \cite[Geometric Lemma]{PRS}.
An idea is to find a curve (or a curve for each $n$) joining $z_1$ to $z_2$ in $\ov\C$ not fast accumulated by $f^n(\Crit f)$, therefore with a controlable distortion for all branches of $f^{-n}$ on its adequate neighbourhoods.

So one concludes the equality between the tree pressures for every non-exceptional $z$ and justifies the definition of tree pressure
\begin{equation}\label{def:tree-pres}
P_{\tree}(f,t):=P_{\tree}(f,t,z)
\end{equation}
for every non-exceptional $z$.

\end{proof}

%{\bf Warning}

\begin{proof}[Proof of Theorem \ref{equality-sup-pressure}, the item 3. ]
%Now we prove the item 3.
The inequality 

\noindent $P_{\tree}^{\rm pullinf}(f,t,z)\le P_{\tree}(t,f, z)$ holds trivially  for every  $z\in \ov\C$ as before
since obviously $\inf\{v\in W_{k,z_k,r}: |f'(y)|^{-1}\}
\le |f'(z_k)|^{-1}$ for all $k$.

\smallskip

So let us prove the opposite inequality for every non-exceptional $z\in J(f)$, namely
$$
P_{\tree}^{\rm pullinf}(f,t,z)\ge P_{\tree}(f,t, z).
$$

\smallskip

We apply for any pullback $\Comp_x f^{-k}(B(f^k(x),\tau))$ the notation $W_{k,x,\tau}$.
Write
 $r_n:=\sup_{z\in J(f), v\in f^{n}(z)}\{\diam W_{n, v,r}\}$ and $r_{\max}:=\max \{r_n: n=1,2,...\}$ for all $n=1,2,... $.
Notice that

(*)  \ for $r\to 0$ we have $r_{\max}\to 0$.

\smallskip

\noindent Indeed, by the backward uniform asymptotic stability, for all $r\le r_0$ as in Definition \ref{back-stab},
$$
\forall\epsilon>0\; \exists n(\epsilon)\; \forall n\ge n(\epsilon)\; \forall W_{n,v,r}\, ,\; \diam W_{n,v,r} \le \epsilon.
$$
Notice also, that for every $\delta>0$ there exists $\delta'>0$ such that for every $z\in \ov\C$ and every component  $\diam (\Comp f^{-1}(B(z,\delta')))\le \delta$. By iterating a number ot times smaller than $n(\epsilon)$ we
conclude that there exists $0<r'\le r$ such that for every
$n<n(\epsilon)$,\; $\diam W_{n,v,r'} \le \epsilon$. Thus,  $(r')_{\max}\le \epsilon$.
This proves (*).

\medskip

Having chosen $z\in J(f)$
%such that $z\notin B(\bigcup_{j=0,...,j(r)}(\Crit(f)),r)$
consider an arbitrary backward trajectory $(z_n)$.
%of $z$, namely a sequence $(z_n)$ of points $z_n\in f^{-n}(z)$ such that $f(z_{n+1})=z_n$ for $n=0,1,...$ and $z_0=z$.
To simplify notation assume that $z$ is non-periodic, hence
$z_n$ determines $n$.

\medskip

Consider since now on $r$  such that $2r_{\rm max}<r_0$.
Define inductively a sequence of integers $k_j$.
Let $k_1$ be the least $k\ge 1$ such that
$$
\widetilde W^{k,1}:=W_{k,z_k,2r}$$
 %containing $z_k$
contains a critical point. Pay attention to the coefficient 2 at $r$; it will guarantee bounded distortion
on discs of radius $r$ of the branches $f^{-k}, k=1,...,k_1 -1$ mapping $z$ to $z_k$.
%Then by (*) $W_k\subset B(z_{k_1}),r)$.

Having defined $k_j$ %and  $\widetilde W^{k_j,j}$
we consider the least $k=k_{j+1}>k_j$
such that
$$
\widetilde W^{k_{j+1},j+1}:=W_{k_{j+1}-k_j,z_{k_{j+1}},2\diam W_{k_j,z_{k_j},r}}
$$ contains a critical point.

\smallskip

Let $C$ be an upper bound of the distortion of $f^{k_{j+1}-k_j-1}$ on
$f( W_{k_{j+1},z_{k_{j+1}},r})$; by Koebe distortion lemma it is universal for all $r\le r_0$, see \cite[Lemma 6.2.3]{PUbook}.
Notice also that
$$
\frac{\diam  W_{k_{j+1},z_{k_{j+1}},r }}{\diam f (W_{k_{j+1},z_{k_{j+1}},r})} \le C  |f'(y)|^{-1}
$$
for all $y\in W_{k_{j+1},z_{k_{j+1}},r}$ for each $j$, and a constant $C$ depending on the maximal degree of criticality at critical points.

\smallskip

Finally notice that since  $\diam  W_{k_j,z_{k_{j}},r}\to 0$,  then by (*)

$$
\diam \widetilde W^{k_{j+1}} \to 0
$$ as $j\to\infty$, uniformly with respect to the backward trajectories $(x_n)$ for $x\in J(f)$.
So $k_{j+1}-k_j \to \infty$ uniformly if the same critical point above appears, since otherwise the trajectory
$(x_n)$ would not be in $J(f)$, see
e.g. \cite[Lemma 1]{P-PAMS}.
Taking in account that  $f$ has  only a finite number of critical points, we conclude that
$\#\{j: k_j\le n\} /n \to 0$.

\

Thus, for every $v\in W_{n,z_n,r}$
\begin{equation}\label{eq:diam-der}
%\prod_{k=1}^{n} \sup\{y\in f^{n-k}(W_{n,z_n,r}): |f'(y)|\}
\Pi_n(1,v)= |(f^n)'(v)|^{-1}     \ge
(\exp -n\epsilon)
\frac{\diam W_{n,z_n,r}}{r}.
\end{equation}
for all $n\ge n(\epsilon)$
uniformly with respect to the backward trajectory $(z_n)$, for all $\epsilon>0$.

\

Finally we shall use the assumption that $z$ is non-exceptional, set $v=z_n$.%, namely $z\notin E$
and prove the inequality roughly opposite to \eqref{eq:diam-der}:
\begin{align}\label{diam-derivative}
\begin{split}  \frac{\diam W_{n,z_n,r}}{r}   & \ge
C\exp -n\epsilon \frac{\diam \Comp_{z_n} f^{-n} (B(z, C\exp -n\epsilon))}
{C\exp (-n\epsilon)} \\ &
\ge
\Const (\exp -n\epsilon) |(f^n)'(z_n)|^{-1} .
\end{split}
\end{align}
%with Comp being the component containing $z_n$.

The latter inequality follows from the the bounded distortion of the conformal mapping
$f^{-n}:B(z,C^{-1}\exp(-n\epsilon))\to \Comp_{z_n} f^{-n} (B(z, C^{-1}\exp -n\epsilon))$  for a constant $C>0$  and every $n$ since $z$ is non-exceptional.

\smallskip

In other words, for all $r>0$ small enough,
\begin{equation}\label{individual}
\lim_{n\to \infty}  \frac1n \log \frac{\Pi^{\rm pullinf}_n(1,z_n)}{|(f^n)'(z_n)|^{-1}}=1,
\end{equation}
the convergence uniform over all backward trajectories of $z$.
\smallskip

Taking power $t$ and summing both inequalities \eqref{eq:diam-der} (taking $\inf_v$ and \eqref{diam-derivative} over $z_n$ for $n\to\infty$ and $\epsilon\to 0$ yields the inequality $\ge$ in Theorem \ref{equality-sup-pressure}, hence the equality.

\end{proof}

Notice that in $P^{\rm pullinf}_{\tree}$ some components $W_{n,z_n,r}$ can contain many elements of $f^{-n}(z)$ thus being counted many times, but the number of these times is upper bounded by $\exp \tau n$ for $\tau$ arbitrarily small and $n$ large,
again due to scarcity of "critical" times: $k_{j+1}-k_j \to \infty$ for each critical point.
This justifies
\begin{definition}\label{def:PsupW}[pullback infimumn W-tree pressure]
\begin{equation}\label{PsupW}
P^{\rm pull W}_{\tree}(f,t,z):=\lim_{r\to 0}
\limsup_{n\to\infty}\frac{1}{n}\log\sum_{W\in {\mathcal W}_n}
\prod_{k=1}^{n} \inf\{y\in f^{n-k}(W): |f'(y)|^{-t}\}.
\end{equation}
%$, , like $P^{\rm pullinf }_{\tree}(f,t, z)$ but with
where $\mathcal{W}_n$ is the family of all pullbacks of $B(z,r)$ for $f^n$,
The limit in \eqref{PsupW} for $r\to 0$ exists due to obvious monotonicity, compare Definition \ref{def:pullinf}.
%$ similarly to $P_{\rm sup}$ but in the sum on \eqref{Psup} replacing the summing over $v\
\end{definition}
Thus we obtain
\begin{corollary}
$P^{\rm pull W}_{\tree}(f,t,z)=P_{\tree}(f,t)$ for every non-exceptional $z$.
\end{corollary}

\begin{remark} The inequality \eqref{eq:diam-der} was proved in \cite{P-Perron} and named a "telescope lemma".
However there the exponential convergence of $(f^n)'(v)|^{-1}$ to 0 was assumed and nothing assumed  on
$\diam W_{n,z_n,r}$. Here the uniform convergence $\diam W_{n,z_n,r}  \to 0$ is a priori assumed,
but nothing about the derivatives.
\end{remark}

\begin{remark}\label{unlike}  In $P^0_{\tree}(f,t)$ a fraction
$\frac {
\prod_{k=0}^n \inf\{|f'(y)|^{-1}: y\in B(f^{n-k}(v),\delta)\} }{
|(f^n)'(z_n)|^{-t}}$
can be very close to 0, unlike in \eqref{tele}.
So an analogon of \eqref{tele} need not hold. See Remark \ref{omit-troubles} and Section \ref{trunk-tree}.

\end{remark}

\

\section{McMullen's pressures}\label{sec:McM}

%--------------------------------

\

Define now McMullen's pressure. Assume there is a puzzle structure for $f$ (a Markov partition with singularities).
Namely there exists a covering  of a neighbourhood of $J(f)$ by a family $\cP$ of closed topological Jordan discs $P_i$, small enough that none of them contains more than one critical point, whose interiors $\Int P_i$ intersect $J(f)$ and are mutually disjoint, and such that if $f(\Int P_i)$ intersects
$\Int P_j$ then $f(\Int P_i)\supset \Int P_j$.  We assume also that all the maps $f|_{\Int P_i}$ are proper.
We allow critical points to belong to the boundaries of $P_i$.

Note that this definition has some differences from McMullen's one. Firstly, we do not use any measure in it.
Secondly, McMullen assumed an expanding property, we do not. On the other hand he assumed $f$ to be continuous (conformal)
only piecewise, that is on each $P_i$.

Following McMullen \cite{McM-HD3} define a refinement ${\sR}({\cP})$
by

\noindent $\cl \Bigl(\Int f^{-1}(\cP) \bigvee {\cP}\Bigr)$ that is the family of the closures of all components of the sets
$f^{-1}( \Int P_j) \bigvee {\Int P_i}$. % for which $f(\Int P_i)$ intersects $\Int P_j$.
This is also a covering of a neighbourhood of $J(f)$ (maybe smaller than the one for $\cP$) by the complete $f$-invariance of $J(f)$, and has a puzzle structure.
%they are open Jordan discs if we assume that each $P_i$ contains at most
We consider consecutive refinements ${\sR}^N({\cP})$ and assume that the diameters of
their elements shrink to 0 uniformly.

\medskip

%For each  $P_{N,j}\in \sR}^N({\cP})$ select a point $y_{N,j}\in \Int P_{N,j}

For $P_{N,i}, P_{N,j}\in \sR^{N}({\cP})$  denote  and number
the closures  of the  components of $\Int P_{N,i})\cap
 \Int f^{-1}(P_{N,j})$ %if this set is non-empty,
 by $P_{N+1,i,j,s}$ where $s=1,...,s(i,j)=s_N(i,j)$.
 %Number them by $P_{N+1,\iota}\in \sR^N({\cP})$.
%The subscripts $s$ stand for integers numbering the componentsof each $\Int P_{N,i})\cap
 %\Int f^{-1}(P_{N,j})$, $s=1,...s(i,j)$.
 Sometimes we shall omit $s$, to simplify notation.
 Of course the number of the components $s(i,j)$ is larger than 1 if and only if
 there is an $f$-critical point
 $c\in \Int P_{N,i}\setminus \Int f^{-1}(P_{N,j})$ (remember that we assume the puzzle pieces are small enough each to contain at most one $f$-critical point in $\Int P_{N.i}$).
 %For each pair $(i,j)$, \; as we assumed  the puzzle pieces
  %to be small enough that each contains at most one critical point, say $c$.
  In the later case $s(i,j)=\nu(c)$ (the multiplicity of $f$ at $c$).
  If the intersection is empty, we set $s(i,j)=0$
  %iff
  %$c\in P_{N,i}$ but $f(c)\notin \Int P_{N,j}$.
 In this notation the number of components $P_{N+1,\iota}\in \sR^N({\cP})$ is $\sum_{i,j} s(i,j)$.

 In McMullen's hyperbolic setting all $s(i,j)$ are 0 or 1 since there are no critical points on stage.
 The same holds in Example \ref{ex:2}.

\medskip
 %\in P_{N,i}\cap f^{-1}(P_{N,j})$.
 %fix an arbitrary non-exceptional point $x^i\in \Int P_i$.

With each  $\cP$ and $N$ as above
distinguish a point $y_{N,i}$ %=y_{N-1,i,j,s}$
 in each $\Int P_{N,i}$.
With $\sR^N(\cP)$ we associate the
%transition
matrix $\sR^N(T)$ with the entries
 %$\sR^N(T)_{ij}=\sum_{y_{N,i,j}\in \Int P_{N,i}  \cap f^{-1}( \Int P_{N,j})}
 \begin{equation}\label{Nentries}
 a_{ij}=
 \begin{cases}
 |f'(y_{N,i})|^{-1} & \mbox{if } \; s(i,j)>0, \\  0 & \mbox{if }\; s(i,j)=0.
 \end{cases}
 %\sum_s(|f'(y_{N,i,j,s})|^{-1}),
  \end{equation}
 % in particular 0 if the intersection above is empty (i.e. $y$ does not exist).
 %We denote the transition matrix for the refinement by$\sR(T)$ and the consecutive refinements by ${\sR}^N(T)$.

\medskip

If $T$ is a rank $M$ matrix, then $\sR(T)$ is a rank $\sum_{i,j=1}^M s(i,j)$ matrix. Similarly $\sR^{N+1}(T)$
for $\sR^N(T)$ in place of $T$.
If we consider a simplified integer-valued rank $M$ matrix $\hat T$, with
  each $ij$ entry equal to $s(i,j)$, then in the directed graph interpretation the vertices  for $\sR(\hat T)$ are the edges
 of the graph of $\hat T$. In the notation above $\iota$ corresponds to an edge $(i,j,s)$. This is a derived graph concept, see \cite{Ore}, except the multiplicities $s(i,j)$ of the edges for $\hat T$, each giving rise to $s(i,j)$
 vertices for the graph of $\sR(\hat T)$.

%(If we did not care about the weights $|f'(y_{N,i,j})|^{-1}$ we would consider a 0-1 matrix in place of $T$ and
%acting with $\sR$ on the related directed graph would mean taking a derived graph, see \cite{Ore}.)

\medskip

%Let $T^N$ for each $N\in\N$ is then the transition matrix for ${\sR}^N({\cP})$.
Let $\lambda(\sR^N(T))$ denote the spectral radius of $\sR^N(T)$.
For each $t>0$ we use the notation replacing above $T$ by $T^t$, the matrix with each entry
being the adequate entry for $T$ raised to the power $t$. Similarly we define ${\sR}^N(T)^t$. In particular we
 denote its spectral radius by $\lambda(\sR^N (T)^t)$.
 %Note that due to the \emph{topological exactness} of $f$ on $J(f)$
 %Suppose that  $f:X\to X$ is \emph{topologically exact}. %, that is for every open $U\subset X$ there exists $m$such that $f^m(U)=X$. .................
Due to the topological exactness of $f$ on $J(f)$
 \begin{equation}\label{entry}
 \lambda ({\sR}^N(T)^t):=\lim_{n\to\infty} {\root n \of {||(({\sR}^N(T))^t)^n||}}=
 \lim_{n\to\infty} {\root n \of {(({\sR}^N (T)^t)^n})_{ij}}
 \end{equation}
 independently of an arbitrary position $ij$ for $1\le i,j \le {\rm rank}(({\sR}^N(T))^t)$.
Indeed,  in our situation the topological exactness says that all the matrices ${\sR}^N(T)$ are \emph{primitive}, that is
 $(({\sR}^N(T))^n)_{ij}>0$ for $n$ large enough for all $i,j$. Useful is the term $(i_0,...,i_n)$ being
 \emph{admissible}, which means that $\sR^N(T)_{i_k i_{k-1}}>0$ for all $k=1,...,n$. Using this we can say
 that ${\sR}^N(T)$ is primitive if for all $n$ large enough and all $i,j$ there exists an admissible sequence
$(i_0,...,i_n)$ such that $i_0=j$ and $i_n=i$.

\smallskip

At last define \emph{McMullen's pressure}, see \cite{McM-HD3} in the hyperbolic case.

\begin{equation}\label{def:McM}
P_{\rm McM}(f,t):=\limsup_{N\to\infty} \log\lambda ({\sR}^N(T)^t).
\end{equation}

\medskip

{\bf Warning.} Unfortunately in presence of critical points in $J(f)$ this notion has deficiencies if the distinguished points
are critical or close to critical ones, making $P_{\rm McM}(f,t)$ too big, bigger than $P(f,t)$. A remedy is to consider

\begin{definition}[restricted McMullen's pressure]\label{def:trunkMcM}
Define the restricted McMullen's pressure as
\begin{equation}\label{eq:restrictedMcM}
\widehat{P_{\rm McM}}(f,t):=\lim_{N\to\infty} \log\lambda (\widehat{{\sR}^N(T)}^t),
\end{equation}
where in each $\widehat{{\sR}^N(T)}$ we consider all the entries at the positions $ij$
such that
\begin{equation}\label{trunk-matrix}
%\frac{\dist (P_{N,i}\cap f^{-1}(P_{N,j}),\Crit(f))}{\diam P_{N,i}\cap f^{-1}(P_{N,j})} \ge A(N)
\frac{\dist (P_{N,i}, \Crit (f)}{\diam P_{N,i}} \ge A(N)
\end{equation}
the same as in ${\sR}^N(T)$, and
all other to be 0. Here $A(N)$ is an arbitrary sequence of numbers tending to $\infty$ as $N\to\infty$
such that $A(N) {\diam }(\sR^N(\cP)) \to 0$,
where diam here denotes
supremum of diameters of the sets of a partition.

The limit in \eqref{eq:restrictedMcM} exists since the ranks of the matrices are growing, acquiring growing number of entries (puzzle pieces), and the puzzle pieces already considered become split with growing $N$, the distinguished points may move but the movements distances decrease to 0. A detailed proof relies on
\eqref{simplified}   and \eqref{details}. In fact we do not need it a priori to prove any limit
(for a convergent subsequence) is equal to $P(f,t)$ as in Proof of Theorem \ref{McM-tree}.
Still a deficiency of this notion is that with $y_{N,i}$ arbitrary, even far from $\Crit (f)$,  there is no reason that the sequence is monotone increasing, nor that its elements do not exceed $P(f,t)$ slightly.

\end{definition}

\medskip

So, consider also suggested in \cite{DGT}, see here Example \ref{ex:2},

\begin{definition}[fuzzy McMullen's pressures]\label{fuzzyMcM}

To define $P^0_{\rm McM}(f,t)$, fuzzy McMullen's pressure (another name: infimum McMullen's pressure),
just replace $|f'(y_{N,i})|^{-t}$ by  $\inf_{y\in P_{N,i}} |f'(y)|^{-t}$
for each $N$ and $P_{N,i}$
in the definition of McMullen's pressure in $a_{ij}$, see \eqref{Nentries}. Namely consider adequate matrices
$(\sR^N(T)^t)^{\inf}$ and their eigenvalues $\lambda_{N,t}^{\inf}$ and

\begin{equation}\label{eq:fuzzyMcM}
P^0_{\rm McM}(f,t):=\lim_{N\to\infty} \log\lambda_{N,t}^{\inf}.
\end{equation}
The limit exists since the sequence is monotone increasing, because when the puzzle pieces split for $N$ growing,
infima are taken on smaller sets.
\end{definition}
%Define also \emph{"fuzzy" close to Crit McMullen's pressure}, $P_{\rm McM}^{0,\Crit}(f,t)$
%by replacements as above only in $P_{N,i,j,s}$ not satisfying \eqref{trunk-matrix}.
%\end{definition}

\begin{definition}[fuzzy restricted McMullen's pressure]  To define it, denoted  $\widehat{P^0_{\rm McM}}(f,t)$, keep unchanged the entries $a_{ij}$ in the matrices accompanying  $P^0_{\rm McM}(f,t)$ for $i$ satisfying  \eqref{trunk-matrix},
putting 0 elsewhere. The monotone increasing holds as before.
\end{definition}

\smallskip

Now we can complete the proof of Theorem \ref{equality-sup-pressure}.

\begin{theorem}\label{McM-tree}
In the setting above
for each $t>0$ %and non-exceptional $x\in J(f)$
$$
P(f, t) = \widehat{P_{\rm McM}(f,t)}=\widehat{P^0_{\rm McM}(f,t)}= P^0_{\rm McM}(f,t) \le P_{\rm McM}(f,t).
$$

%:=\lim_{N\to\infty}\log\lambda(({\sR}^N)^t)= P_{\sup}(f.t,x)= P_{\tree}(f,t,x).
\end{theorem}

\begin{proof} We prove the following. For any non-exceptional $z\in J(F)$
\begin{align}\label{ineq}
\begin{split}
P_{\hyp}(f,t) \le  &
 \widehat{P_{\rm McM}}(f,t) =
\widehat{P_{\rm McM}^0}(f,t) \le \\ &
 P^0_{\rm McM} (f,t) \le
P_{\tree}(f,t,z) \le
P_{\hyp}(f,t).
\end{split}
\end{align}

Consider an arbitrary set $X\in\cH(f,J(f))$. Let $N$ be large so that
every $P_{N,i}\in \sR^N(\cP)$ has diameter less than $\delta\ll\dist(X,\Crit(f)$, as in \eqref{trunk-matrix},
so that $\frac{|f'(x)|}{|f'(v)|}$ for $v\in B(x,\delta)$ for each $x\in X$ is close to 1 (compare Proof of Theorem
\ref{hatP0}).
Consider now only the puzzle pieces intersecting $X$.
Thus, passing to limits we obtain $P_{\hyp}(f,t) \le \widehat{P_{\rm McM}}(f,t)$.

Similarly,  considering puzzle pieces
$P_{N,i}$ satisfying \eqref{trunk-matrix} and distinguished $y_{N,i}$ in them, we  prove  the equality
$\widehat{P_{\rm McM}}(f,t)= \widehat{P^0_{\rm McM}}(f,t)$. Clearly only the $\le$ part is non-trivial.
Here is the proof:

For any $v\in P=P_{N,i}$ we have, close to a critical point $c$ of multiplicity $\nu$,
assuming for simplification that $f(x)=(x-c)^\nu$, writing $y$ for $y_{N,i}$
\begin{equation}\label{simplified}
\Bigl|\Bigl(\frac{|f'(v)|}{| f'(y)|}\Bigr)^{1/(\nu-1)} - 1\Bigr| \le  \Bigl|\Bigl(\frac{\dist (v,c)}{\dist (y,c)}\Bigr) -1\Bigr| \le
\end{equation}
$$
\frac {|\dist(v,c) - \dist(y,c)|}  {\dist(y,c)} \le
\frac {\diam(P)}{\dist (P,c)}\le A(N)^{-1},
$$
see \eqref{trunk-matrix}. So $\frac{|f'(v)|}{| f'(y)|}\to 1$
 as $N\to\infty$.
Without the simplification we can write $f(z)=a(z-c)^\nu(h(z))$ for an analytic map $h(z)=1+a_1 (z-c) + ...$
in a neighbourhood of $c$, so
\begin{align}\label{details}
\begin{split}
f'(z) & =  a \nu (z-c)^{\nu-1} h(z) + a (z-c)^\nu h'(z) \\ & =
a (h(z)\nu + (z-c)h'(z))(z-c)^{\nu-1} \\ &  =
a \bigl(1+a_1 (z-c) +...)\nu + (z-c)(a_1+ ...)\bigr)(z-c)^{\nu-1}   \\ & =
a\nu(z-c)^{\nu-1} (1 + O(z-c)).
\end{split}
\end{align}

So for $\dist(P,c)$ small enough, for all $N$, $|f'(v)/f'(y)|\le 2 A(N)^{-1}$. Further from $c$,
i.e. in a domain bounded away from $\Crit(f)$,
$\log |f'|$ is Lipschitz continuous, so for large
$N$,\   $\dist(v,y)$ implies $|\log |f'(v)|-\log |f'(y)||$ small, so $|f'(v)/f'(y)|$ is close to 1,
namely it tends to 1 as $N\to\infty$.

\medskip

%This is a estimate more precise than the preceding one proving the first inequality.
%However it would be sufficient here to prove that $\frac{|f'(v)|}{| f'(z)|})
%\frac{|f'(v)|}{| f'(z)|})$ and its reciprocal are bounded by a constant not depending on $N$.

%\begin{equation}
%\log |{f^n}'(z_n)| - \log \prod_{k=1}^n |f'(y_{N,i_k,i_{k-1}})|     \le \Const n\delta'
%\end{equation}

The inequality $\widehat{P^0_{\rm McM}}(f,t) \le P^0_{\rm McM}(f,t)$ is
obvious because the matrices associated to the former one have just zeros replacing non-zero terms in the latter ones.

%Similarly one proves $P_{\hyp}(f,t) \le
%$$P^{0,\Crit}_{\rm McM}(f,t) \le P^0_{\rm McM}(f,t)$$

\

To prove the inequality $P^0_{\rm McM}(f,t)\le P_{\tree}{\tree}(f,t)$, first for each $P_{N,i}\in \sR^N(\cP)$ select an arbitrary   non-exceptional point $z^{N,i}$ in it.
Consider an arbitrary sequence of integers $i_0,...,i_n$, admissible, %with the multiplicities $s_1,...,s_n$,
that is such that %$s_{i_k}$ denotes the number of components of  $
$\Int P_{N,i_k}\cap f^{-1}(P_{N,i_{k-1}})\not=\emptyset$ for all $k=1,...,n$.
It contributes in the matrix $({\sR}^N(T)^t)^n$ as corresponding to the path in the related directed graph with edges joining consecutively the vertices $i_n,...,i_0$.
%with multiplicities $s_n,...,s_1$.

Consider
$z_n\in \bigcap_{k=0}^n f^{-(n-k)}(\Int P_{N,i_k})$  (maybe disconnected !)   %(one for each components)
and $z_k:=f^{n-k}(v_n)$ for all $k=n-1,...,0$, with common
$z=z_0=                        %f(y_{N,i_1,i_0})$ is $z
z^{N,i_0}$.                   % distinguished in $\Int P_{i_1}$ for $P_{i_1}\in\cP$.
In fact the number of possible points $z_n$ is equal to
$\prod_{k=0}^{n-1} \deg(f|_{P_{N,k+1}\cap f^{-1}(P_{N,k})})$.

Now $P^0_{\rm McM} (f,t) \le P_{\tree}(f,t,z)$ follows from the obvious
$\inf_{y\in P_{N,i_k}} |f'(y)|^{-1}\le |f'(z_k)|^{-1}$ for each $k=1,...,n$.
Indeed the only issue is that we took care only for
the sequences $(i_0,...,i_k)$ with common $i_0$.
Considering all the sequences the constant factor $\#(\sR^N(\cP))$ appears.
However it disappears when $n\to \infty$ in the definition of the spectral radius.

%Now we start using the definition of $\widehat{{\sR}^N(T)}^t$. We have for $\delta'=O(\delta)$
%\begin{equation}
%\log |{f^n}'(z_n)| - \log \prod_{k=1}^n |f'(y_{N,i_k,i_{k-1}})|     \le \Const n\delta'
%\end{equation}

%So, summing over all admissible $i_0,...,i_n$ with the same $P_{N,i_0}$ containing $z$,using \eqref{entry}, taking limits for $n\to \infty$ and $\delta\to 0$ (so $N\to\infty$), we get
%$\widehat{P_{\rm McM}}(f,t) \leP_{\tree}(f,t,z)$.
%(compare again Proof of Theorem \ref{hatP0}).

%Notice that here, unlike in Theorem \ref{equality-sup-pressure} in the item 3, we need not take care abouta possibility that $\Int P_{N,,i_k}\cap f^{-1}(P_{N,i_{k-1}}$ contains more than one $f$-preimage of $z_{k-1}$
%since $f$ is injective on it as not containing any critical point. In particular it is connected so all multiplicities $s_1,...,s_n$ for admissible sequences $i_0,...,i_n$ are eaual to 1, so these indexes are redundant.

\smallskip

The last inequality in \eqref{ineq}  is known if $z\in J(f)$ is hyperbolic and non-exceptional, with a proof via capturing hyperbolic subsets via shadowing; it was mentioned in Proof of Theorem \ref{equality-sup-pressure}, referring to \cite{PRS}.

Finally, notice that  the inequality, say,  $P_{\rm McM}^0(f,t)\le P_{\rm McM}(f,t)$ is obvious.

\end{proof}

\begin{remark}[Fuzzy multiple McMullen's pressure]
Notice that replacing the sequence $y_{N,i_k}$ for admissible $(i_0,...,i_n)$ we do not exploit all appropriate
$z_n$. So in the definition of the matrix $\sR^N(T)$ we could consider each entry $a_{ij}$ multiplied by
$ \deg(f|_{P_{N,i}}))=s_N(i,j)$. For a related notion of pressure, to be called
\emph{fuzzy multiple MsMullen's pressure}, denoted $P^0_{\rm multMcM}(f,t)$ we also get the  upper bound
by $P_{\tree}(f,t)$.
\end{remark}

\begin{remark}\label{omit-troubles}

%We are not able to prove the inequality $P_{\rm McM}(f,t)\le $

%\noindent $P_{\tree}(f,t,z)$
%even if we take care for an appropriate choice of selected non-exceptional points $z=z^{N,i}$
%$y_{N,i.j}$,
%since

We cannot prove
for each sequence $i_1,...,i_n$ admissible for the matrix $\sR^N(\cP)$ %for a non-exceptional $z$
neither the
inequality
\begin{equation}\label{eq:1}
\liminf_{N\to\infty}\liminf_{n\to\infty} \frac1n \log\Bigl(
\frac {\prod_{k=1}^n  |f'(y_{N,i_k})|}
{|(f^n)'(z_n)|} \Bigr)
\le 0
\end{equation}
nor the opposite one with liminf replaced by limsup,
%latter inequality for each backward trajectory of $z$ separately,
%more precisely to prove $\lim_{n\to\infty} \frac {\log\Pi_n^{\inf}(1,z_n)}{\log\Pi_n(1,z_n)}=1$,
unlike in e.g.
\eqref{individual}. A problem is, that the expressions
 $|(f^n)'(z_n)| $                                                           %$\Pi^{\inf}_n(1,z_n)$ there,
must be replaced by
$\prod_{k=1}^n   |f'(y_{N,i_k })|$, where domains of $f'$ to which $y_{N,i_k }$ respectively belong, are
larger than the components of $f^{-k}(B(f^n(z_n),r)$ for large $k$, so the latter products can be too large or too small.
Namely the choices of $y_{N,i_k}$ for $P_{N,i_k}$ close to $\Crit(f)$, are much further from $\Crit(f)$ than $z_k$, or much closer.

%(too close to $\Crit(f)$ or much further than $) .
A way out of this trouble would be to replace one telescope in
the proof of \eqref{individual} by a sequence of telescopes. But for this, to prove e.g. $\le 0$ in
\eqref{eq:1} we need to know that each $z_{n_j}$ starting a new telescope is non-exceptional (with the same constants) to obtain  \eqref{diam-derivative}, which can be impossible.

So, a way we have chosen to avoid $\prod_{k=1}^n   |f'(y_{N,i_k })|$ too large, has been just to get rid of the trouble-making backward trajectories, by considering the restricted McMullen's pressure.

A way to avoid $\prod_{k=1}^n   |f'(y_{N,i_k })|$ too small, that is $\prod_{k=1}^n   |f'(y_{N,i_k })|^{-1}$
too large, is to consider  fuzzy McMullen's pressure
replacing $|f'(y_{N,i})|^{-1}$ by $\inf \{|f'(y)|^{-1}: y\in P_{N,i}\}$, see Definition \ref{fuzzyMcM}. Without this it can just happen that
$P_{\rm McM}(f,t)>P(f,t)$, see the Warning preceding Definition \ref{def:trunkMcM} (or the restricted one as in that Definition).
For another remedy, replacing the infimum or one point $y_{N,i}$ by pairs of distinguished points, see Section \ref{final}.
\end{remark}

%\begin{remark} Notice that the ch3oice of the distinguished points in McMullen's pressure, $y_{N,i,j}\in \sR^N\cP$ does not matter,if we just put 0 in place of $f'(y_{N,i,j})^{-1}$ in $P_{N,i}$ containing, or close to, a critical point.

%A conclusion is that to approximate $P(t)$ it is sufficient to calculate $f'$ only at $f^n$-preimages of a chosen non-exceptional $z$, being bounded away from $\Crit f$ and localized in a bounded precision.
%\end{remark}

\begin{remark}\label{XN}
Notice that each set
\begin{equation}
X_N=\bigcap_{n=1}^\infty \bigcup_{i_0,...i_n}\bigcap_{k=0}^{n-1} f^{-(n-k)}(P_{N,i_k})
\end{equation}
the summation over all $(i_0,...,i_k)$ admissible for the restricted matrix, i.e. $\widehat{\sR^N(T)}$, see
\eqref{trunk-matrix}.

\smallskip

Notice that each $X_N$ is hyperbolic for $f|_{X_N}$. Indeed, inverses of

\noindent $(f|_{\Int P_{N,i_0,...,i_n}})^n$
where $P_{N,i_0,...,i_n}:={\bigcap_{k=0}}^{n-1} f^{-(n-k)}(P_{N,i_k})$,
for each $i_n$ is a
Montel normal family of holomorphic maps on $P_{N,i_n}$. This is so, because the ranges omit $f^k(\Crit(f)), k=0,...,n$, i.e. eventually more then 2 points in $\ov\C$. So their limits must be points since otherwise a limit domain $U$ would not intersect $J(f)$ as all $f^n(U')$ for some open $U'\subset U$ intersecting $J(f)$ and $n$ large enough, are bounded in $P_{N,i_n}$, contradicting the definition of Julia set.
On the other hand $U$ must intersect $J(f)$ by its backward invariance and compactness. This implies uniform
convergence of $|(f^n)'|^{-n}$ on $X_N$ to 0, hence hyperbolicity.

\smallskip

Notice however that $X_N$ need not be repelling, like in Definition \ref{hyperbolic pressure},
equivalently: the maps $f|_{X_N}$ need not be open. See \cite[Example 4.5.5]{PUbook} where  a
question of a small extension of $X_N$ to an invariant set on which $f$ is open was discussed.

\smallskip

Notice finally that the sequence of the sets $X_N$ is monotone increasing with respect to inclusion
and that the pressures satisfy
$$
P(f|_{X_N}, -t\log |f'|)= \log \lambda (\widehat{\sR^N(T)}^t).
$$
%Compare Remark \ref{monotone}.
\end{remark}

\section{Restricted fuzzy tree pressure}\label{trunk-tree}

\medskip

One more definition of pressure might be useful for computations, close to the restricted McMullen's pressure and to the fuzzy tree pressure as in \eqref{P0} and \eqref{Pdelta}, for the potentials $-t\log|f'|$, making sense for
all rational maps. Namely define

\begin{equation}\label{convergence}
\widehat{P^0_{\tree}} (f,t,z):=\lim_{\Delta\to 0}\widehat{P^{\Delta}_{\tree}} (f,t,z),
\end{equation}
where
\begin{equation}
 \widehat{P^{\Delta}_{\tree}} (f,t,z):=\limsup_{n\to\infty}\frac1n \log
 \sum_v \widehat{\Pi^\delta}(v),
 \end{equation} where the sum is over  all
 $v\in f^{-n}(z)$ such that $\dist (f^{n-k}(v),\Crit (f))>\Delta$ for all $1\le k\le n $.

where
\begin{equation}\label{hatv}
\widehat{\Pi^\delta}(v):= \prod_{k=1}^n |f'(\widehat{ v_k})|^{-t}
\end{equation}
where $\widehat{v_k}$ is a point in $\cl B(f^{n-k}(v),\delta)$ where $|f'|^{-1}$ takes infimum, and
\begin{equation}
\delta=o(\Delta).
\end{equation}

Existence of the limit in \eqref{convergence}  follows from the monotone increasing, which is obvious since the infima are taken on shrinking sets. Notice that if we do not mind about the monotonicity, we can choose
$\widehat{v_k}$ arbitrarily (randomly) in the ball.

\smallskip

\begin{theorem}\label{hatP0}
 $\widehat{P^0_{\tree}} (f,t,z)$ does not depend on non-exceptional $z$. So (omitting writing $z$) we have
$$\widehat{P^0_{\tree}} (f,t)= P(f,t)$$
\end{theorem}

\begin{proof}
The proof is similar to the proof of Theorem \ref{McM-tree}.
The inequality
$\widehat{P^0_{\tree}} (f,t,z)\le P^0_{\tree}(f,t,z)$ is obvious, just more backward branches of $z$ in the latter pressure are considered. Also $P^0_{\tree}(f,t,z) \le P_{\tree}(f,t,z)$ is obvious, it was mentioned already in
Theorem \ref{equality-sup-pressure}, the item 1.

It is left to prove $P_{\hyp}(f,t)\le \widehat{P^0_{\tree}} (f,t,z)$.
For this, consider an arbitrary hyperbolic $X\subset J(f)$. It is enough to
consider a non-exceptional point $z\in X$ (or just
a preimage under an iterate of $f$, arbitrarily close to $X$,
of an priori given point $z$).
We prove $P_{\tree}(f|_X,t,z) \le \widehat{P^0_{\tree}} (f,t,z)$.
It is a repetition of the proof of $P_{\hyp}(f,t) \le
 \widehat{P_{\rm McM}}(f,t)$ in Theorem \ref{McM-tree}.
%the part
%$\widehat{P_{\rm McM}}(f,t) \le
%\widehat{P_{\rm McM}^0}(f,t)$.
For $\Delta=\dist (X,\Crit)$ we use
$\delta=o(\Delta)$ to assure that for $x\in X$ and $y\in B(x,\delta)$ for $x,y$ close to a critical point $c \in \C$ with a multiplicity $\nu$ for $f$, the ratio
$$
|f'(x)|/|f'(y)|\le \Const \bigl(|x-c|/|y-c|)\bigr)^{\nu-1} \le 1+ \Const \Bigl(\frac{\delta}{\Delta}\Bigr)^{\nu-1}
$$
is close to 1. For more details see \eqref{details}.  In other words the difference of the potentials $-t\log |f'(x)| - (-t\log |f'(y)|)$ is small.

\end{proof}

%The same observation is meaningful also while proving Theorem \ref{McM-tree}.

\begin{remark} Introducing $\widehat{P^0_{\tree}}$ between $P_{\hyp}$ and $P^0_{\tree}$ shows directly  how to omit the problem for individual backward trajectories, see Proof of
Theorem \ref{equality-sup-pressure} the item 1 in Section \ref{sec:fuzzy-inf}, Remark \ref{unlike}, and compare  Remark \ref{omit-troubles}.
\end{remark}

\smallskip

\section{Final remarks, more geometric pressures and examples}\label{final}

\subsection{On convergence}

% Notice that the family of functions $t\mapsto \widehat{P^{\Delta}_{\tree}} (f,t,z)$ are locally equicontinuous, because $t\mapsto |f'(\hat{v_k}|^{-t}$ are. Hence the convergence in \eqref{convergence} is uniform (locally).

\smallskip

\begin{remark}
It is obvious that the sequence of functions $t\mapsto \widehat{P^{\Delta}_{\tree}} (f,t,z)$ in \eqref{convergence} converges uniformly (locally) as $\Delta\to 0$ and clear that the limit $P(f,t)$ is non-increasing, by e.g. the $P_{\hyp}(f,t))$ definition.
So calculating these functions
and their first zeros we  obtain  as the limit the first zero of $P(f,t)$, which is $\HD_{\hyp}(J(f))$, see
Proposition \ref{hypdim}.
Unfortunately we do not know the speed of the convergence for general $f$. For some classes of maps $f$ the situation is better, i.e. topological Collet-Eckmann maps, Remark \ref{CE}.
\end{remark}

\begin{remark}
It might be worthy to use instead, the functions
$$t\mapsto \log \lambda(\widehat{\sR^N(T)}^t),
$$ as $N\to\infty$, and their zeros.

Note that their zeros can be calculated as solutions of the equation $\lambda(\Lambda^t)=1$ if all the entries of
a primitive matrix $\Lambda$ are nonnegative, here for $\Lambda=\widehat{\sR^N(T)}$, see \cite[Practical considerations]{McM-HD3}. See also Remark \ref{XN}.% and \ref{monotone}.

%Notice also that for each $N$ the set $X_N$ being the closure of the invariant set consisting of all points for which all forward trajectories have admissible sequence of symbols with respect to the matrix $\widehat{\sR^N(T)}$, namely ....
%is hyperbolic i.e. in   ... see ...
%\cite{Mane}. So in fact we do not consider more than the sets included in the definition of $P_{\hyp}(f,t}$.

\end{remark}

\begin{remark}\label{CE}
 It may happen that  $\HD_{\hyp}(J(f))<\HD(J(f))$, see \cite[Subsection 2.13.2]{Lyu}, so the methods here are not adequate to estimate $\HD(J(f))$, unless e.g. $f$ is \emph{topological Collet-Eckmann}, see e.g. \cite{P-ICM}, where $P(f,t)$ has only one zero, denote it  $t_0$, and $\HD_{\hyp}(J(f))=\HD(J(f))$.% and even a box dimension of $J(f)$, where the latter is always a zero of $(f,t)$, see ... \cite{Bishop}.

Notice that in this case at $t_0$ the derivative $dP/dt (t_0)<0$ so the convergence of approximations, say $\HD(X_N)\to t_0$ is faster than if the left derivative of  $dP(f,t)$ at $t_0$ were 0.

\end{remark}

\

{\bf Conclusions and considerations.} As noted in Section \ref{Introduction}, our aim is to  approximate the geometric pressure
$P(f,t)$ from below by quantities depending on a parameter $\delta$ for tree or $N$ for McMullen's pressures.
If approximating quantities exceed $P(f,t)$, if we do not know how far are the quantities from the limit,
we do not know how big might an error be in our estimates of $P(f,t)$ from below. In these estimates we use
$P(f,t)=P_{\tree}(f,t,z)$. %, see \eqref{gtp}, where $\limsup_n$ has been used.
To be safe we want also the numbers
under the $\limsup_n$, see \eqref{gtp}, be as small as possible. To this end we can choose any $z $ close to $J(f)$ but outside it, compare \cite{DGT}, hence not only non-exceptional, but not accumulated by forward trajectories  of critical points at all.
Notice that if $z_1,z_2$ are like $z$ and belong to the same component $B$ of Fatou set, then all the ratios
$\Pi_n(t,v_1)/\Pi_n(t,v_2)$ for corresponding $v_i\in f^{-n}(z_i), i=1,2$ are uniformly bounded. Coresponding  in the sense that for a curve $\gamma$ joining $z_1$ to $z_2$ in $B\setminus {\rm PC}(f)$ each $v_1$ and $v_2$ are the end points of a lift of $\gamma$ for $f^n$.
This happens for polynomials, where both $z_i$ belong to the basin of $\infty$.

As we noted in Section \ref{Introduction},
among the notions of appropriate geometric pressures we introduced
to approximate (calculate) $\HD(J(f))$ from below, the infimum (fuzzy) and/or restricted pressures are appropriate, in particular
$\widehat{P_{\rm McM}^0}(f,t)$,  $P_{\rm McM}^0(f,t)$, $P_{\rm multMcM}^0(f,t)$,
$\widehat{P^0_{\tree}}(f,t,z)$ and $P^0_{\tree}(f,t,z)$ might occur useful, since elements of the sequences defining them, depending on $\delta$ or $N$, do not exceed $P(f,t)$.

 %Of course also $\widehat{P_{\tree}}(f,t,z)$ is OK, since there is no $\delta$ (fuzzyness) in its Definition \ref{def:trunkMcM}.

% Restricted, means in particular that we do not consider all backward branches of $z$ or not all the puzzle-pieces in the consecutive
% generations, but omit those close to $\Crit(f)$.

Another pressure: $\widehat{P_{\rm McM}}(f,t)$ is "almost" monotone increasing,
%(similarly to double sampling below),
because of bounded distortion in the puzzle pieces satisfying \eqref{trunk-matrix}.
This distortion, responsible for possible decreasing shrinks to 0, as $N\to\infty$ with the speed depending on
$A(N)$.

Some
%, considering some of them
may increase the speed of approximation, but may lead to results exceeding $P(f,t)$, as it may happen with $P_{\rm McM}(f,t)$, see \eqref{def:McM}. This is so because of the use of
distinguished points where $|f'|$ can be too small (its inverse too large). See Remark \ref{omit-troubles}.
On the other hand, considering $P^{\rm pullinf}_{\tree}(f,t)$
%or $P^0_{\rm McM}(f,t)$
requires finding infima in sets shrinking with the time of iteration,
which might be computationally awckward.

\subsection{Double sampling pressures}

A remedy to avoid the quantities exceeding $P(f,t)$ and an excessive complexity of calculations would be something between, e.g. double (or multiple) - sampling variants.

\begin{definition}\label{def:samples-tree} Define \emph{double sampling tree pressure} $P_{\tree}^*(f,t,z)$ similarly to $P_{\tree}^0(f,t,z)$
but replacing infima by minima over two points, namely %being the limit for  $\delta\to 0$ of
\begin{equation}\label{limit*}
P_{\tree}^*(f,t,z):=\limsup_{\delta\to 0} P_{\tree}^{*,\delta}(f,t,z), \ \ {\rm where }
\end{equation}
\begin{equation*}
P_{\tree}^{*,\delta}(f,t,z):= \limsup_{n\to\infty} \frac1n \log
%P_{\tree}^{*,\delta}(f,t,z,n)$$
%with P_{\tree}^{*,\delta}(f,t,z,n):=
\sum_{v\in f^{-n}(z)}\Pi_n^{*\delta}(v),
\end{equation*}
where
$$
\Pi_n^{*\delta}(v):=\prod_{k=1}^n
\min (|f'(v_{k,1})|^{-t}, |f'(v_{k,2})|^{-t}),
$$
where $v_{k,1}, v_{k,2}\in B(f^{n-k}(v),\delta)$, are symmetric to each other with respect to
$v_k=f^{n-k}(v)$ and $\dist(v_{k,1},v_{k,2})\ge \delta$. Compare \eqref{hatv}.
Thus in place of one distinguished point in  $B(v_k,\delta)$, we choose two.
\end{definition}

If there is a critical point $c$ close to $v_k$, then at least  one $v_{k,i}$ is further from $c$ than $v_k$.
Even $|f'(v_{k,i})|\ge f'(v_k)|$, hence $|f'(v_{k,i})|^{-t}\le f'(v_k)|^{-t}$. At other points $v_k$, for which $\delta=o(\dist (v_k,\Crit(f)))$, this inequality holds up to a factor $1+\epsilon$, where $\epsilon=O(\delta)$.
%If $\delta\ll \dist(v_k,\Crit(f)$, then
Then $|f'(v_{k,i})|^{-t} / |f'(v_k)|^{-t} \approx 1$ for $i=1,2$.
So
\begin{equation}\label{ratio}
\frac{\Pi_n^{*\delta}(v)}{|(f^n)'(v)|^{-t}}\le 1+\epsilon,
\end{equation}
 So, taking $n\to\infty$ next summing over $v$ and taking $\delta\to 0$ we obtain
\begin{proposition}  For every non-exceptional $z$, \;
$P_{\hyp}(f,t)\le P_{\tree}^*(f,t,z)\le P_{\tree}(f,t,z)$.
\end{proposition}

Unfortunately we cannot prove the monotone increasing of $P^{*,\delta}_{\tree}(f,t,z)$ as $\delta\to 0$,
nor $P_{\tree}^{*,\delta}(f,t,z) \le P_{\tree}^{*}(f,t,z)$.
For $v_k$ close to a critical point $c$, if $\dist(v_k,c)=C\delta$, we have
$\dist (v_{k,i},c)\ge \delta \sqrt{C^2+1/2}=
\dist(v_k,c) \frac{\sqrt{C^2+1/4}}{C}$ for $i=1$ or 2.
Hence
$$
\frac{|f'(v_{k,i})|^{-t}} { f'(v_k)|^{-t}}\le  \Bigl(\frac{C}{\sqrt{C^2+1/4}}\Bigr)^{t\nu(c)} (1+O(\delta)).
$$
However  we cannot achieve this gain far from $\Crit(f)$. A remedy would be to consider
\emph{triple sampling tree pressures}, with $v_{k,i}, i=1,2,3$, at the vertices of a equilateral triangle
centered at $v_k$. Then $f'$ in a small neighbourhood of $v_k$ is almost affine, so the inequality
$|f'(v_{k,i})|\ge f'(v_k)|$ holds for some $i$ provided there is no point $x_0$ where $f''(x)=0$.
If the latter case takes place, assume $f'(x)=a(x-x_0)^{m(x)} + ...$ for $a\not=0$ and an integer
$ m(x)>1$. To cope with this case if it happens, consider
\emph{m-sampling tree pressure}, with $v_{k,i}, i=1,...,m$ at the vertices of a regular $m$-gon centered at $v_i$. Then, for $m=3\max\{m(x): f''(x)=0\}$, there exists $i$ such that $|f'(v_{k,i})|\ge |f'(v_k)|$.
So we can skip $\epsilon$ in \eqref{ratio} and then conclude with
$
P_{\tree}^{*,\delta}(f,t,z) \le P(f,t)
$ for each $\delta$ for $m$-sampling tree pressure.

\medskip

\begin{definition}[double sampling McMullen's pressure]
Similarly we define \emph{double sampling McMullen's pressure} $P_{\rm McM}^*(f,t)$.
For %$P= P_{N,i,j,s}=\cl \Comp \Int P_{N,i}\cap f^{-1}(P_{N,j})$
$P=P_{N,i}$ for which
$\Int P_{N,i}\cap f^{-1}(\Int P_{N,j})\not=\emptyset$
we consider two points $v_{P,1},v_{P,2}\in
B(P, r_{N,P})$ where $r_{N,P}:=A(N)(\diam P))$ for an arbitrary sequence
$A(N)\to\infty$ as
$N\to\infty$, but  $A(N) \diam ( {\sR}^N(\cP))\to 0$, compare \eqref{trunk-matrix}.
The points $v_{P,1}, v_{P,2}$ are chosen
symmetric with respect to an arbitrary point $z^*_P$ in $P$ and

\noindent $\dist (v_{P,1},v_{P,2}) \approx r_{N,P}$ (i.e. far from $P$ compared to its diameter).
We distinguish such a  pair  only for $P$ such that $\dist (P,\Crit(f))\le r_{N,P}$. In this case we need to do so because of the arbitrariness of the choices of $z^*_P$.

%Recall that mesh denotes supremum of diameters of the sets of a partition.
%, where $\delta(N)\to 0$ as $N\to\infty$ but $\delta(N)\gg \diam P$. Formally the ball has its origin at an arbitrary point in $P$ and the symmetry is with respect to this point.

Now define the matrices ${\sR}^N(T)^*$ by changing in ${\sR}^N(T)$ defined at the beginning of
Section \ref{sec:McM} the entries $|f'(y_{N,i})|^{-1}$
with distinguished points $y_{N,i}$ to $a_{ij}=\min (  |f'(v_{P,1})|^{-1},     |f'(v_{P,2})|^{-1}       )$ for $P$ close to $\Crit(f)$ as above.
Finally define
\begin{equation}\label{def:*}
P_{\rm McM}^*(f,t)=\lim_{N\to\infty} \log \lambda ((  \sR^N(T)^*)^t).
\end{equation}
\end{definition}

\medskip

Consider now an arbitrary non-exceptional $z\in J(f)$ not belonging to  the boundary of any puzzle piece of any generation. Then for each $z_k\in f^{-k}(z) \in P_{N,i}$, we have $a_{ij}^t\le |f'(z_k)|^{-t}  (1+\epsilon)$,
where $\epsilon \to 0$ as $N\to\infty$, compare \eqref{ratio}.
So indeed it holds:

\begin{proposition}\label{prop:P*MCM}
$P_{\hyp}(f,t)\le P_{\rm McM}^*(f,t)\le P_{\tree}(f,t,z)=P_{\tree}(f,t)$.
\end{proposition}
%The proof is as for $P_{\tree}^*(f,t,z)$.

%There is no reason that the sequence  $\lambda((\sR^N(T^*))^t)$ in \eqref{def:*} is monotone increasing as $N\to\infty$, but for subsequences it is, compare Definition \ref{def:samples-tree}.

\begin{definition}\label{sample-in-P}
In fact we could distinguish $v_{P,1}, v_{P,2}\in \Int P$. Indeed, assume $d:=\dist(v_{P,1}, v_{P,2})\ge \frac12 \diam P$. So $\dist (z_k, v_{P,\iota})\le 2d$ for $z_k\in P$ for both $\iota=1$ and $\iota=2$. But $\dist (c,  v_{P,\iota})\ge d/2$
for $\iota =1$ or $\iota=2$ for each point $c$ close to $P$, in particular a critical one. So for such $\iota$
we get due to the triangle inequality, skipping indices,
$ \dist(c,z)\le \dist (c,v) + \dist( v, z)$
$$
\frac{  \dist(c,z)} {\dist(c,v)} \le 1 + \frac{\dist (v,z)}{ \dist( c,v)} \le 1 + \frac {2d}{d/2}\le 5.
$$
So, for $\nu$ denoting the multiplicity of $f$ at $c$,
$$
\frac { |f'(z_k)|^{-1}}{|f'(v_{P,\iota})|^{-1}} \ge \Const 5^{-(\nu-1)}, \ \ \ {\rm hence}
\ \ \ |f'(v_{P,\iota})|^{-1} \le \Const^{-1} 5^{\nu -1} |f'(z_k)|^{-1}.
$$

%$\dist (c,v_{P,\iota}) \ge \frac{\diam P}{4} \ge\frac{\dist(z_k,  v_{P,\iota})}{4}$.

%So $|f'(v_{P,\iota})|^{-1} \le \Const |f'(z_k)|^{-1}$.
\noindent Since this occurrence happens rarely with $N$ large the constant $\Const^{-1} 5^{\nu -1}$ does not matter.

\end{definition}

Note that the shape of $P$ can be very distorted, making finding above $v_{P,1}, v_{P,2}\in \Int P$ difficult.
So instead we can consider in the definition taking $v_{P,1}, v_{P,2}\in B(P, \diam P)$.

Unfortunately these constructions do not allow to prove neither monotonicity nor that the elements
of the sequence in \eqref{def:*} do not exceed $P(f,t)$ (though discrepancies seem low), unless we modify the definition to an
\emph{m-sampling McM-pressure}, as in the tree pressure case in Definition \ref{def:samples-tree}.
\medskip

%\begin{remark}\label{monotone}
%The sequences $P_{\tree}^{*,\delta}(f,t,z)$  and $\lambda (({\sR}^N(T)^*)^t)$
 %converging to $P (f,t)$ are only roughly (provided the symmetric distinguished points are sufficiently apart from each other), monotone increasing as $\delta\to 0$ and $N\to\infty$ respectively,
 %and larger than the ones in the restricted notions. So they may be useful to fasten the speed of convergence. \end{remark}

   % \begin{remark} Suppose that defining $P^*_{\rm McM}(f,t)$ above one keeps
   %  $a_{ij}$ above only if $\Crit (f)$ intersects $B(P, \delta(N))$. Otherwise  put %$a_{ij}=|f'(y_{N,i,j})|^{-1}$,

     %to
    %$$
    %{\hat a}_{ij}:=\imn (  |f'(v_{P,1})|^{-1},     |f'(v_{P,2})|^{-1} , |f'(y_{N,i,j})|^{-1}      )
    %$$

    %with $y_{N,i,j}\in P$ as in the definition of McMullen's pressure in Section \ref{sec:McM}.
    %Then, proving Proposition \ref{prop:P*MCM}, for a given non-exceptional $z$ and its backward trajectory $(z_n)$, if $z_n\in P$ and $\dist(\Crit (f) , P ) \ge \delta(N)$ corresponding to $A(n)(\diam P)$ in the definition of the
    %restricted McMullen's pressure, see \eqref{trunk-matrix}, the pair of the distinguished points $v_{P,1}, v_{P,2}$ does not appear. The points $y_{N,i,j}$ are sufficient.
    %This shows that $ \widehat{P_{McM}}(f,t) \le P_{\rm McM}^*(f,t)$.

    %, since in the the former all $P$ considered are like that

    %It is just said here that distinguishing the auxiliary pair of points is not needed if we consider only restricted McMullen's pressure.

    %\end{remark}

\subsection{Examples}

\begin{example}\label{ex:1} For each non-renormalizable polynomial $f$, say with connected Julia set and all periodic orbits repelling, one considers Yoccoz puzzle construction, namely a covering $\cP$ of a neighbourhood of $J(f)$
 whose pieces have boundaries consisting of equipotential lines  for the Green's function in the basin of $\infty$ and external radii to fixed points dissecting $J(f)$. Assume these points are not postcritical. Diameters of consecutive pullbacks of these pieces shrink uniformly to 0, so the assumptions of Theorem \ref{equality-sup-pressure}, item 2, are satisfied.  See e.g. \cite{KvS} for this and more general cases.
 %See Figure 1.
 \end{example}

\begin{remark} There is no need in the definition of McMullen's pressures that the consecutive
puzzle structures  are of the form $\sR^N(\cP)$. One can just take any sequence $\cP_N$ of puzzle structure coverings such that diameters of their elements tend uniformly to 0. Then in Example 1 one may allow infinitely renormalizable case, with so-called \emph{a priori complex bounds} condition, allowing to find such a sequence, see \cite{Lyu} and references therein.
\end{remark}

\begin{example}\label{ex:2} For the Feigenbaum map $f_{\rm Feig}(z)=z^2+c_{\rm Feig}$ where $c\approx -1.40155$,
infinitely renormalizable, where $c_{\rm Feig}$ is the limit of the decreasing sequence of the period doubling real parameters, a different puzzle structure is used, see \cite{DS} and references therein. The critical point 0
is in the boundary of four first generation puzzle pieces adjacent at it, so all restrictions of $f_{\rm Feig}$
to all generations puzzle pieces are injective, the integer parameter $s$ is not needed, see beginning of Section \ref{sec:McM}. In fact in
\cite{DGT} it is announced that for $t=\HD_{\hyp}(J(f))$ it holds $\log\lambda^0_{N,t}\to \delta^0_{\rm cr}$ as $N\to\infty$, with $\lambda^0_{N,t}$ defined in our
Definition \ref{fuzzyMcM} and $\delta_{\rm cr}$ being the critical exponent of Poincar\'e series. The latter
is equal to $\HD(J(f))$ and Minkowski dimension (box dimension) of $J(f)$ provided the area of $J(f)$ id zero
(which is the case by \cite{DS} for $f_{\rm Feig}(z)$),
for the latter see \cite{Bishop}. (There, Witney critical exponent appears, but it is straightforward that it is
Minkowski dimension of $J(f)$, by Koebe's $1/4$ lemma. In \cite{Bishop}, actions by Kleinian groups were  considered.)

%starting from
%$F:W\C \setminus (-\infty,-\frac1{\lambda] \cup [\frac1{lambda^2},\infty))$ where $F$ is the Feigenbaum function fixed under a renormalization with the rescaling coefficient $\lambda$, with $\cP$ having two pieces, with interiors being  the upper and the lower half-planes.

\noindent See \cite{DS}, Figures 1,2 and a precise description there.

%Feigenbaum maps

%-- H. Epstein,

%-- A. Dudko \& S. Sutherland \cite{DS}

%-- A. Dudko, I. Gorbovickis, W. Tucker

\end{example}

%$\lambda(T^t) \le \epsilon$

%Conclusion on $HD_{\hyp}$

%\noindent {\bf Topological and geometric pressures -- glossary}

%$P_{\var}(f,\phi)$, \ \ \ variational pressure, \eqref{varpres}.

%$P_{\rm {sep}}(f,\phi)$, \ \ \  topological pressure on separated sets, Definition \ref{topological pressure}..

%$P^0_{\sep}(f,\phi)$, \ \ \ fuzzy pressure, \eqref{fuzzy}.

%$P_{\tree}(f,\phi,z)$, \ \ \ tree pressure, \eqref{Ptree}.

\smallskip

%$P^0_{\tree}(f,\phi,z)$, \ \ \ fuzzy tree pressure, \eqref{P0}.

%\smallskip

%$P_{\tree}(f,t,z)$, \ \ \ geometric tree pressure, \eqref{gtp}.

%\smallskip

%$P^{\rm pullinf}_{\tree}(f,t,z)$, \ \ \ pullback infimum tree pressure, \eqref{Pi-inf}

%\smallskip

%$P_{\hyp}(f,t)$, \ \ \ hyperbolic pressure, Definition \ref{hyperbolic pressure}

%\smallskip

%$P^{\rm{pull W}}_{\tree}(f,t,z)$, \ \ \ pullback infimum W-tree pressure, \eqref{PsupW}

%\bibliographystyle{amsplain}

%\begin{section*}[geometric pressures]

%\end{section*}

\end{document}